\documentclass[12pt,twoside]{amsart}
\usepackage{amsmath}
\usepackage{amsthm}
\usepackage{amsfonts}
\usepackage{amssymb}
\usepackage{latexsym}
\usepackage{mathrsfs}
\usepackage{amsmath}
\usepackage{amsthm}
\usepackage{amsfonts}
\usepackage{amssymb}
\usepackage{latexsym}
\usepackage{geometry}
\usepackage{dsfont}
\usepackage[dvips]{graphicx}
\usepackage{color}
\usepackage[all]{xy}

\date{}
\allowdisplaybreaks[4] \footskip=15pt
\renewcommand{\uppercasenonmath}[1]{}

\usepackage{graphicx,amssymb}
\usepackage[all]{xy}
\usepackage{amsmath}

\allowdisplaybreaks
\usepackage{amsthm}
\usepackage{color}

\theoremstyle{plain}

\theoremstyle{plain}
\newtheorem{theorem}{Theorem}[section]
\newtheorem{proposition}[theorem]{Proposition}
\newtheorem{lemma}[theorem]{Lemma}
\newtheorem{corollary}[theorem]{Corollary}

\newtheorem*{open question}{Open Question}
\newtheorem{definition}[theorem]{Definition}

\theoremstyle{definition}
\newtheorem*{acknowledgement}{Acknowledgement}

\theoremstyle{remark}
\newtheorem{remark}[theorem]{Remark}

\newcommand{\A}{\mathcal{A}}

\newcommand{\Q}{\mathcal{Q}}

\def\GV{{\rm GV}}
\def\tor{{\rm tor}}
\def\Hom{{\rm Hom}}
\def\Ext{{\rm Ext}}

\def\E{{\rm E}}

\def\Nil{{\rm Nil}}

\def\H{{\rm H}}
\def\Z{{\rm Z}}
\def\GV{{\rm GV}}
\def\Max{{\rm Max}}
\def\DW{{\rm DW}}

\def\T{{\rm T}}
\def\E{{\rm E}}
\def\dep{{\rm depth}}
\def\DQ{{\rm DQ}}
\def\Spec{{\rm Spec}}
\newcommand{\m}{\frak{m}}
\newcommand{\p}{\frak{p}}
\newcommand{\q}{\frak{q}}
\newcommand{\M}{\frak{M}}
\newcommand{\Pp}{\frak{P}}
\newcommand{\Qq}{\frak{Q}}
\newcommand{\I}{I}
\newcommand{\J}{J}
\def\Min{{\rm Min}}

\begin{document}
\begin{center}
{\large  \bf  $q$-Krull dimension of commutative rings}

\vspace{0.5cm}  Xiaolei Zhang$^{a}$\\

{\footnotesize a.\ School of Mathematics and Statistics, Shandong University of Technology, Zibo 255000, China\\

 E-mail: zxlrghj@163.com\\}
\end{center}

\bigskip
\centerline { \bf  Abstract}
\bigskip
\leftskip10truemm \rightskip10truemm \noindent

In this paper, we introduce and study the $q$-Krull dimension of a commutative ring via its $q$-operation. A new characterization of $\tau_q$-von Neumann regular rings is obtained, and some properties of rings  $q$-Krull dimension 0 are studied. Moreover, we characterize  the $q$-Krull dimension of $\tau_q$-Noetherian rings in terms of Krull dimension of some total quotient rings. Besides, we show that every ring with $w$-Krull dimension $0$ is a \DW\  ring, positively answers an question proposed by H. Kim.
\vbox to 0.3cm{}

\leftskip0truemm \rightskip0truemm
\bigskip

\section{Introduction}
Throughout this paper, we always assume $R$ is a commutative ring with identity. For a ring $R$, we denote by $\T(R)$ the total quotient ring of $R$, $\Min(R)$ the  minimal prime spectrum of $R$, $\Nil(R)$ the nil radical of $R$.

Many classical rings and modules can be generalized by using star operations. $w$-operation is one of the important operation which is studied by many algebraists (see \cite{fk16} for a summary). 
In 2020, Zhou et al. \cite{ZDC20} introduced the notion of $q$-operations over commutative rings utilizing finitely generated semi-regular ideals.  Since every nonzero ideal over integral domain is semi-regular,  $q$-operations over integral domains are trivial. Based on this viewpoint, $q$-operations shed a little light in the area of commutative rings with zero-divisors.  
Recently, Zhang and Qi \cite{ZQ23} and Zhang \cite{zq-ps} introduced the $\tau_q$-semisimple rings, $\tau_q$-von Neumann regular rings and $\tau_q$-coherent rings, which can be seen as a generalization of semisimple rings, von Neumann regular rings and coherent rings using using $q$-operations. In particular, A ring is $\tau_q$-semisimple ring if and only if it is reduced with finite minimal prime spectrum, and a ring is $\tau_q$-von Neumann regular if and only if it is reduced with compact minimal prime spectrum.

As our work involves $q$-operations, we give a brief introduction on them. For more details, refer to  \cite{ZDC20}. Let $R$ be a ring and $A,B\subseteq R$. Denote by $(A:_RB):=\{r\in R\mid Br\subseteq A\}$.
Recall that an ideal $I$ of $R$ is said to be \emph{dense} if $(0:_RI)=0$, and \emph{semi-regular} if there exists a finitely generated dense sub-ideal of $I$. The set of all finitely generated semi-regular ideals of $R$ is denoted by $\Q$. Let $M$ be an $R$-module. Denote by
 \begin{center}
{\rm $\tor_{\Q}(M):=\{x\in M|Ix=0$, for some $I\in \Q \}.$}
\end{center}
Recall from \cite{wzcc20} that an $R$-module $M$ is said to be \emph{$\Q$-torsion} (resp., \emph{$\Q$-torsion-free}) if $\tor_{\Q}(M)=M$ (resp., $\tor_{\Q}(M)=0$). The class of $\Q$-torsion modules is closed under submodules, quotients and direct sums, and the class of $\Q$-torsion-free modules is closed under submodules and direct products. A $\Q$-torsion-free module $M$ is called a \emph{Lucas module} if $\Ext_R^1(R/J,M)=0$ for any $J\in \Q$, and the \emph{Lucas envelope} of $M$ is given by
\begin{center}
{\rm $M_q:=\{x\in \E_R(M)|Ix\subseteq M$, for some $I\in \Q \},$}
\end{center}
where $\E_R(M)$ is the injective envelope of $M$ as an $R$-module.
By  \cite[Theorem 2.11]{wzcc20}, $M_q=\{x\in \T(M[x])|Ix\subseteq M$, for some $I\in \Q \}.$ A $\Q$-torsion-free module $M$ is called \emph{$\tau_q$-finitely generated} provided that there is a finitely generated submodule $N$ of $M$ such that $N_q=M_q$.
Obviously, a $\Q$-torsion-free $M$ is a Lucas module if and only if $M_q=M$. A \emph{$\DQ$ ring} $R$ is a ring for which every $R$-module is a Lucas module.
By \cite[Proposition 2.2]{fkxs20}, $\DQ$ rings are exactly rings with small finitistic dimensions equal to $0$.
Recall from \cite{ZDC20} that an submodule $N$ of a  $\Q$-torsion free module $M$ is called a \emph{$q$-submodule} if $N_q\cap M=N$. If an ideal $I$ of $R$ is a $q$-submodule of $R$, then $I$ is also called a \emph{$q$-ideal} of $R$. Follow from \cite[Proposition 2.7]{ZDC20}, a prime ideal $\p$ is a $q$-ideal if and only if $\p_q\cap R\not=R$, if and only if $\p$ is a non-semiregular ideal.  A \emph{maximal $q$-ideal} is an ideal of $R$ which is maximal among the $q$-submodules of $R$. The set of all maximal $q$-ideals is denoted by $q$-$\Max(R)$, and it is the set of all maximal non-semi-regular ideals of $R$, and thus is non-empty and a subset of $\Spec(R)$ (see
\cite[Proposition 2.5, Proposition 2.7]{ZDC20}).

\section{The  $q$-Krull dimension of Rings}

Let $R$ be a commutative ring. It is well-known that the Krull dimension of $R$ is the supremum of the heights of all maximal ideals of $R$. Since  a prime ideal is a $q$-ideal if and only if it is a non-semiregular ideal, every prime ideal contained in a prime  $q$-ideal is also a prime  $q$-ideal. Hence, we can give a new dimension by considering   the supremum of the heights of all maximal $q$-ideals.

\begin{definition}
Let $R$ be a ring. The $q$-Krull dimension of $R$ is defined to be
\begin{center}
	$q$-$\dim(R)=\sup\{ht \p\mid\p\in q$-$\Max(R)\}.$
\end{center}
\end{definition}

From the definition, the $q$-Krull dimension of $R$ is the supremum of lengths of chains of $q$-prime ideals.  So $q$-$\dim(R)\leq \dim(R);$ and  if $R$ is a $\DQ$-ring, then $q$-$\dim(R)=\dim(R).$

Recall from \cite{ZQ23} that a ring $R$ is called \emph{$\tau_q$-von Neumann regular} if $R_{\m}$ is a von Neumann regular ring for any $\m\in q$-$\Max(R)$, which is equivalent to that for any finitely generated ideal $K$ of $R$, there exists an $I\in\Q$ such that $IK=K^2$; or
 $R_{\m}$ is a von Neumann regular ring for any $\m\in q$-$\Max(R)$; or
  $\T(R[x])$ is a von Neumann regular ring;
or   $R$ is a reduced ring and $\Min(R)$ is compact (see \cite[Theorem 4.9]{ZQ23}). The following result give a new characterization of $\tau_q$-von Neumann regular rings.
\begin{theorem}\label{ncqvon}
	A ring $R$ is a $\tau_q$-von Neumann regular ring if and only if $R$ is reduced with $q$-$\dim(R)=0$.
\end{theorem}
\begin{proof}
	Let $R$ be a $\tau_q$-von Neumann regular ring. Then  $R$ is reduced by \cite[Theorem 4.9]{ZQ23}. Let $\p$ be a prime  $q$-ideal of $R$. Then there exists a maximal $q$-ideal $\m$ of $R$ such that  $\p\subseteq \m$. By localizing at $\m$, we have $\p_\m=\m_\m$ since $R_\m$ is a local von Neumann regular ring which is a field  	 by \cite[Theorem 4.9]{ZQ23}. By localizing at other maximal $q$-ideal $\q$ of $R$, we also have  $\p_\q=\m_\q(=R_\q)$. Hence
	$\p=\p_q=\m_q=\m$ by \cite[Proposition 2.7(6)]{ZDC20}. Consequently, $q$-$\dim(R)=0$.
	
	On the other hand, suppose   $R$ is a reduced ring with $q$-$\dim(R)=0$. Let $\m$ be a maximal $q$-ideal of $R$, then $ht(\m)=0$. So $R_\m$ is a reduced ring with $\dim(R_\m)=0$. It follows from \cite[Remark,Page 5]{H88}  that $R_\m$ is a von Neumann regular ring. Hence $R$ is $\tau_q$-von Neumann regular ring by \cite[Theorem 4.9]{ZQ23}.
\end{proof}

\begin{proposition}\label{nil}
	Let $R$ be a ring and $\Nil(R)$ the nil-radical of $R$. Then $q$-$\dim(R)\geq q$-$\dim(R/\Nil(R)).$
\end{proposition}
\begin{proof} 
We only need to show that if $\p/\Nil(R)$ is a non-semiregular prime ideal of $R/\Nil(R)$, then  $\p$ is a non-semiregular prime ideal of $R$. On contrary, suppose $\p$ is a semiregular prime ideal of $R$. Then there is a finitely generated semiregular subideal $J$ of $\p$. We claim that $(J+\Nil(R))/\Nil(R)$ is a finitely generated semiregular subideal of $\p/\Nil(R)$. Indeed, suppose $(J+\Nil(R))/\Nil(R)\cdot(r+\Nil(R))=0+\Nil(R)$. Then $J r\subseteq \Nil(R)$. Since $J$ is finitely generated, $J^nr^n=0$ for some integer $n$. Note that $J^n$ is also semiregular, so $r\in \Nil(R)$, and thus the claim holds.
\end{proof}

\begin{proposition}\label{qd0}
	Let $R$ be a ring with $q$-$\dim(R)=0$. Then the following statements hold.
	\begin{enumerate}
		\item $R/\Nil(R)$ is a $\tau_q$-von Neumann regular ring.
		\item the minimal spectrum $\Min(R)$ is compact.
	\end{enumerate}
\end{proposition}
\begin{proof}
	The results follow by Theorem \ref{ncqvon} and Lemma \ref{nil}.
\end{proof}

Recall from \cite{H88} that a ring $R$ has \emph{property $\A$} if each finitely generated ideal $I\subseteq \Z(R)$ has a nonzero annihilator, or equivalently,  every finitely generated semiregular ideal of $R$ is regular.  Integral domains, Noetherian rings, nontrivial graded rings (e.g. polynomial rings), rings with Krull dimension equal to $0$ and  Kasch  rings (i.e., rings with only one dense ideal) are rings having property $\A$. By \cite[Corollary 2.6]{H88}, a ring $R$ is an $\A$-ring if and only if so is $\T(R)$. Recall from \cite[section 3]{G89} that the associated prime ideal of a ring $R$ is a prime ideal which is minimal over $(0:_Rr)$ for some $r\in R.$ It follows by \cite[Theorem 3.3.1]{G89} that the set of zero-divisors of a ring is  precisely the union of all   associated prime ideals. So we have the following corollary.

\begin{corollary}
	Let $R$ be a ring  having \emph{property $\A$}. Then 
	\begin{center}
	$q$-$\dim(R)=
\sup\{ht(\p)\mid \p\ \mbox{is an associated prime ideal of}\ R\}$.
	\end{center}
\end{corollary}

\begin{remark}\label{example} Note that the converse of Proposition \ref{qd0} is not true and the inequality in Proposition \ref{nil} can be taken strictly  in general. Let $R=k[[x,y_1,\dots,y_n]]/\langle x^2,xy_i\mid i=1,\dots,n\rangle$ where $k$ is a field and $n$ is a positive integer. Then $\dim(R)=n$ and the minimal  spectrum $\Min(R)$ is finite. Set $\m:=\langle \overline{x},\overline{y_i}\mid i=1,\dots,n\rangle$. Then $\m$ is an associated prime ideal of $R$ with $ht(\m)=n$. So $q$-$\dim(R)=n$. However, since $R/\Nil(R)$ is a reduced $\tau_q$-von Neumann regular ring, we have $q$-$\dim(R/\Nil(R))=0$.
\end{remark}

Recall from \cite[section 6.8]{fk16} that the $w$-Krull dimension $w$-$\dim(R)$ of a commutative ring $R$ is defined to be the supremum of the height of all maximal $w$-ideals. The $w$-Krull dimension of some integral domains are studied in \cite{fk16,w99,w01} . It is well-known that rings with weak $w$-global dimension $0$ are rings with weak
global dimension $0$; rings
with global weak $w$-projective dimension $0$ are rings with global
dimension $0$ (see \cite{wk14,wq19}). Note that both of them
are \DW\ rings. H. Kim proposed the following question in 2023 January 12-13 at ``the International E-Conference of Young Researchers in
Algebra and Number Theory'' held at Morocco:

\begin{center}
\textbf{Is every ring with $w$-Krull dimension $0$ a \DW\  ring?}
\end{center}
Note that if  an integral domain has $w$-Krull dimension $0$, then it is naturally a field by \cite[Theorem 7.2.12]{fk16}, and so is a \DW\  ring. Now we can show it is generally correct.

\begin{theorem}
Let $R$ be a ring with $w$-$\dim(R)=0$. Then  $\dim(R)=0$, and so $R$ is \DW\  ring.
\end{theorem} 
\begin{proof} On contrary assume that $\dim(R)\geq 1$. Then there is a prime ideal $\p$ of $R$ with $ht\p=1$. Since $w$-$\dim(R)=0$, $\p$ is not a $w$-ideal. And so $\p_w=R$ by \cite[Theorem 6.2.9]{fk16}. Then there is a $\GV$-ideal $J:=\langle x_1,\dots,x_n\rangle$ of $R$ such that $J\subseteq \p$. Note that $\E$-\dep $(J,R):=\inf\{n\mid\Ext_R^n(R/J,R)\not=0\}\geq 2$. Write the sequence $\textbf{x}=\{x_1,\dots,x_n\}$. Set $h^{-}(\textbf{x},R):=\inf\{n\mid\H^n(\textbf{x},R)\not=0\}$, where $\H^n(\textbf{x},R)$ is $n$-th homology of complex $\Hom_R(K_\bullet(\textbf{x}),R)$ with $K_\bullet(\textbf{x})$ the Koszul complex.  It follows by \cite[Theorem 6.1.6]{S90} that $h^{-}(\textbf{x},R)=\E$-\dep $(J,R)\geq 2.$ It follows by \cite[Lemma 3.2]{AT09} that $h^{-}(\textbf{x},R)\leq ht(\langle \textbf{x}\rangle)$,  height of the  ideal generated by $\textbf{x}$. So $ht(J)\geq 2$, and hence $ht(\p)\geq 2$, which is a contradiction.
\end{proof}

\section{$q$-dimension of $\tau_q$-Noetherian rings}
Let $R$ be a ring. We denote by $R[x]$ the polynomial ring over $R$ with one variable. Let $f\in R[x]$, we write $c(f)$ the content of $f$, that is, the ideal generated by the  coefficients of $f$. Note that $f$ is a non-zero-divisor in $R[x]$ if and only if $c(f)$ is a semiregular ideal of $R$. In this section, we always write $q$ for the $q$-operation of $R$ and $Q$ for  the $q$-operation of $R[x]$.

\begin{lemma}\label{DMa}
Let $M$ be a $\Q$-torsion-free $R$-module, $g$ a non-zero-divisor in $R[x]$, and $f\in M[x]$. Then $c(gf)_q=c(f)_q$.
\end{lemma}
\begin{proof}
Since $g$ is a non-zero-divisor in $R[x]$, so is $g^k$, and thus $(c(g)^{k})_q=R_q,$ for any positive integer $k$.	By Dedekind-Mertens formula (see \cite[Theorem 1.7.16]{fk16}), there is a positive integer $k$ such that $c(g)^{k+1}c(f)=c(g)^{k}c(gf)$. Hence $(c(g)^{k+1}c(f))_q=(c(g)^{k}c(gf))_q=((c(g)^{k})_q(c(gf))_q)_q=(R_q(c(gf))_q)_q=((c(gf))_q)_q=(c(gf))_q.$
\end{proof}

\begin{lemma}\label{sing}
	Let $\I$ be an ideal of $R[x]$. If $R\subseteq c(\I)_q$, then there is a $g\in \I$ such that $R\subseteq c(g)_q$.
\end{lemma}
\begin{proof}
	Since $c(\I)\subseteq R\subseteq c(\I)_q$, we have $R_q=c(\I)_q$, and thus $c(\I)$ is  $\tau_q$-finitely generated. So there is $g_1,\dots,g_n\in \I$ such that $c(g_1,\dots,g_n)_q=c(\I)_q$. By \cite[Proposition 1.7.15]{fk16}, there exists $g\in \I$ such that $R\subseteq c(\I)_q=c(g_1,\dots,g_n)_q=c(g)_q.$
\end{proof}

\begin{proposition}\label{poly}
	Let $F$ be a Lucas $R$-module, $F[x]$ a Lucas $R[x]$-module and $M$ a submodule of $F$. Then the following statements hold.
	\begin{enumerate}
		\item If $M$ is a Lucas $R$-module, then $M[x]$ is a Lucas $R[x]$-module.
		\item $M[x]_Q=M_q[x]$.
	\end{enumerate}	
\end{proposition}
\begin{proof}
(1)	Let $\I$ be a (semi)regular ideal of $R[x]$ and $f\in F[x]$ such that $\I f\subseteq M[x].$ Then there exists $g\in \I$ such that $c(g)$ is a finitely generated semiregular ideal of $R$. Since $gf\in M[x]$. Then $c(f)\subseteq c(f)_q=c(gf)_q\subseteq M$ by Lemma \ref{DMa}. And thus $f\in M[x]$. Therefore  $M[x]$ is a Lucas $R[x]$-module.
	
(2) By (1), we have  $M[x]_Q\subseteq M_q[x]$. For any $u\in M_q$, there is  semiregular ideal $\J$ such that $\J u\subseteq M$. And so $\J[x]u\subseteq M[x]$. Note that $\J[x]$ is a regular ideal of $R[x]$. So $u\in (M[x])_Q$. Therefore $M[x]_Q=M_q[x]$.
\end{proof}

\begin{proposition}\label{max}
 	Let $\M$ be a prime ideal of $R[x]$ and set $\p=\M\cap R$. Then $\M$ is a maximal $q$-ideal of $R[x]$ if and only if $\M=\p[x]$ and $\p$ is a maximal $q$-ideal of $R$.
\end{proposition}
\begin{proof}
	Assume that $\M$ is a maximal $q$-ideal of $R[x]$. Claim that $\M=c(\M)[x]$, and so $c(\M)=\M\cap R=\p$. Clearly, $\M\subseteq c(\M)[x]$. On the other hand, if $\M\subsetneq c(\M)[x]$, then $R[x]\subseteq(c(\M)[x])_Q$. So, by Proposition \ref{poly},  $(c(\M)_q\cap R)[x]=(c(\M)[x])_Q\cap R[x]=R[x].$ Thus we have $c(\M)_q\cap R=R.$ By Lemma \ref{sing} there exists $g\in \M$ such that $c(g)$ is a finitely generated semiregular ideal of $R$. So $\M$ is not a $q$-ideal of $R[x]$ which is a contradiction. And so $\M=c(\M)[x]=\p[x]$. By Proposition \ref{poly}, $\p$ is a maximal $q$-ideal of $R$.
	
	Suppose $\M=\p[x]$ where $\p$ is a maximal $q$-ideal of $R$. Let $\Qq$ be a maximal $q$-ideal of $R[x]$ that contains $\M$. Write $\q=\Qq\cap R$. Then $\p\subseteq\q$. By the proof of necessary, we have $\Qq=\q[x]$, where $\q$ is a maximal $q$-ideal of $R$. So $\p=\q$, and thus $\Qq=\M$. Consequently, $\M$ is a maximal $q$-ideal of $R[x]$. 
\end{proof}

\begin{theorem}\label{poly}
Let $R$ be a ring. Then $q$-$\dim(R)\leq q$-$\dim(R[x])\leq 2q$-$\dim(R)$.
\end{theorem}
\begin{proof} Assume $q$-$\dim(R)=n$. Let $\p_n\subsetneq \p_{n-1}\subsetneq\cdots\subsetneq \p_1\subsetneq \p_0=\m$ be a chain of prime $q$-ideals of $R$ with $\m$ a maximal $q$-ideal of $R$. Then $$\p_n[x]\subsetneq \p_{n-1}[x]\subsetneq\cdots\subsetneq \p_1[x]\subsetneq \p_0[x]=\m[x]$$ is a chain of prime  $q$-ideal of $R[x]$ by Proposition \ref{max}. And so $q$-$\dim(R)\leq q$-$\dim(R[x])$.
	
On the other hand, assume $q$-$\dim(R[x])=m$.  Let $\Pp_m\subsetneq \Pp_{m-1}\subsetneq\cdots\subsetneq \Pp_1\subsetneq \Pp_0=\M$ be a chain of prime  $q$-ideal of $R[x]$. Then $$\Pp_m\cap R\subseteq \Pp_{m-1}\cap R\subseteq\cdots\subseteq \Pp_1\cap R\subseteq \Pp_0\cap R=\M\cap R\ \ \ \ \ \ \ \ \ \ \ \  (\star)$$ is a chain of prime  $q$-ideal of $R$  by Proposition \ref{max}. It follows by \cite[Theorem 1.8.15]{fk16} that  different terms in the $(\star)$ are at least $\frac{m}{2}$. So  $\frac{m}{2}\geq n$. Thus $q$-$\dim(R[x])\leq 2q$-$\dim(R)$.	
\end{proof}

Recall from \cite{ZDC20} that a ring $R$ is called to be a \emph{$\tau_q$-Noetherian ring}  if it satisfies the ACC on its $q$-ideals. Following from \cite{zq-ps,ZDC20} that a ring $R$ is a $\tau_q$-Noetherian ring if and only if every ($q$-)ideal is $\tau_q$-finitely generated if and only if $T(R[x])$ is a Noetherian ring if and only if $T(R)$ is a Noetherian ring.

\begin{lemma}\label{N-loc}
Let $R$ be a $\tau_q$-Noetherian ring and $\p$ a prime $q$-ideal of $R$. Then $R_\p$ is a Noetherian ring.
\end{lemma}
\begin{proof} Let $I_\p$ be an ideal of  $R_\p$ with $I$ an ideal of $R$.  Since  $R$ is a $\tau_q$-Noetherian ring, $I$ is $\tau_q$-finitely generated. So there exits a finitely generated ideal $J$ of $R$ such that $I_q=J_q$. And so $I_\p=J_\p$ since $\p$ is a prime $q$-ideal of $R$. Hence  $I_\p$ is finitely generated. It follows that $R_\p$ is a Noetherian ring.
\end{proof}
\begin{remark}
	The converse of Lemma \ref{N-loc} is not true in general. In fact, let $R$ be a von Neumann regular ring not semisimple. Then $R$ is a $\DQ$-ring, and so every prime ideal  is a $q$-ideal. Note that $R_\p$ is a field for every prime ideal. But $R$ is not a Noetherian ring, so is also not $\tau_q$-Noetherian.
\end{remark}

\begin{proposition}\label{Noepol}
Let $R$ be a $\tau_q$-Noetherian ring and $\Pp$ a prime $q$-ideal of $R[x]$. If $(\Pp\cap R)[x]\not=\Pp$, then $ht\Pp=ht(\Pp\cap R)+1$. Hence $ht\Pp=ht(\Pp\cap R)$ if and only if $\Pp=(\Pp\cap R)[x]$.
\end{proposition}
\begin{proof}
Let $\p=\Pp\cap R$. If $\p[x]\not=\Pp$, then $\p R_{\p}[x]\not=\Pp  R_{\p}[x]$. Since  $R$ is a $\tau_q$-Noetherian ring, $R_\p$ is a Noetherian ring by Lemma \ref{N-loc}. Note that $\p R_{\p}=\Pp  R_{\p}[x]\cap R_\p,$ we have $ht\Pp  R_{\p}[x]=ht(\Pp  R_{\p}[x]\cap R_\p)+1$ by \cite[Theorem 4.4.7]{fk16}. Hence $ht\Pp=ht(\Pp\cap R)+1.$
\end{proof}

\begin{lemma}\label{tprime}
	Let $\p$ be a prime $q$-ideal of $R$. Then $\T(\p[x])$ is a prime ideal of $\T(R[x])$.
\end{lemma}
\begin{proof}
It follows that	$\T(R[x])/\T(\p[x])\cong \T(R/\p[x])$ is an integral domain.
\end{proof}

\begin{theorem}
Let $R$ be a $\tau_q$-Noetherian ring. Then 
\begin{center}
$q$-$\dim(R)=q$-$\dim(R[x])=\dim(T(R[x])).$
\end{center}	
\end{theorem}
\begin{proof}
Let $\M$ be a maximal $q$-ideal of $R[x]$. Then by Proposition \ref{max}, there exists a 	maximal $q$-ideal $\m$ of $R$ such that $\M=\m[x]$. So by Proposition \ref{Noepol}, $ht(\M)=ht(\m[x])=ht(\m[x]\cap R)=ht(\m)$. Thus $q$-$\dim(R)=q$-$\dim(R[x])$. 
	
Let $\p$ be a prime $q$-ideal of $R$. Then $\T(\p[x])$ is a prime ideal of $\T(R[x])$ by lemma \ref{tprime}. So $q$-$\dim(R)\leq \dim(T(R[x])).$ On the other hand, suppose $\M$ is a maximal ideal of $\T(R[x])$. Then there is a maximal $q$-ideal $\m$ of $R$ such that $\M= \T(\m[x])$ by \cite[Proposition 3.9]{ZDC20}. Note that $\T(R[x])_{\T(\m[x])}\cong R[x]_{\m[x]}\cong R_\m[x]_{\m R_\m[x]}$ since every element in $\m[x]$ is a zero-divisor. Since $R_\m$ is a Noetherian ring, $ht(\M)=\dim(\T(R[x])_{\T(\m[x])})=ht(\m R_\m[x])=ht(\m R_\m)=ht(\m)\leq q$-$\dim(R)$ by \cite[Theorem 4.4.7]{fk16}. Hence,  $q$-$\dim(R)=\dim(\T(R[x])).$
\end{proof}

\begin{remark}
It follows by Theorem \ref{poly} that $q$-$\dim(R)\leq q$-$\dim(R[x])\leq 2q$-$\dim(R)$ for any ring $R$. For any pair of integers $(n,m)$ of integers with $1\leq n\leq m\leq 2n$, it is  interesting to construct rings with $q$-$\dim(R)=n$ and $q$-$\dim(R[x])=m$. 
\end{remark}

\begin{acknowledgement}\quad\\
The first author was supported by National Natural Science Foundation of China (No. 12061001). 
\end{acknowledgement}

\end{document}